\documentclass{amsart}
\usepackage{amssymb,amsmath,amsfonts,amsthm,textcomp,mathrsfs,microtype}
\usepackage{subdepth}

\newtheorem{lemma}{Lemma}
\newtheorem{thm}{Theorem}[section]

\newtheorem{cor}[thm]{Corollary}

\newtheorem{qst}[thm]{Question}

\theoremstyle{definition}
\newtheorem{dfn}[thm]{Definition}
\newtheorem{eg}[thm]{Example}

\theoremstyle{remark}

\renewcommand{\emptyset}{\varnothing}

\newcommand{\Z}{\mathbb Z}

\newcommand{\C}{\mathsf C}

\renewcommand{\phi}{\varphi}

\renewcommand{\i}[1]{\mathfrak{#1}}

\newcommand{\m}{\i{m}}

\newcommand{\X}{\mathsf X}

\newcommand{\A}{\mathsf A}
\newcommand{\B}{\mathsf B}
\newcommand{\E}{\mathsf E}
\renewcommand{\S}{\mathsf S}

\newcommand{\depth}{\mathop{\mathrm{depth}}\nolimits}

\newcommand{\reg}{\mathop{\mathrm{reg}}\nolimits}

\newcommand{\pd}{\mathop{\mathrm{pd}}\nolimits}

\newcommand{\socle}{\mathop{\mathrm{socle}}\nolimits}

\renewcommand{\*}{\bullet}

\title[Ideals with Large Projective Dimension]{Ideals with Larger Projective Dimension and Regularity}
\author[Beder, McCullough, N\'u\~nez-Betancourt, Seceleanu, Snapp, Stone]{Jesse Beder, Jason McCullough, Luis N\'u\~nez-Betancourt, \\ Alexandra Seceleanu, Bart Snapp, Branden Stone}
\address{Jesse Beder, University of Illinois at Urbana-Champaign, Department of Mathematics 1409 W. Green Street Urbana, IL 61801}
\email{beder@math.uiuc.edu}
\address{Jason McCullough, Department of Mathematics, Univeristy of California - Riverside, 202 Surge Hall, 900 University Ave., Riverside, CA 92521}
\email{jmccullo@math.ucr.edu}
\address{Luis N\'u\~nez-Betancourt, University of Michigan, Department of Mathematics, 3068 East Hall, 530 Church Street, Ann Arbor, MI, 48104-1043}
\email{luisnub@umich.edu}
\address{Alexandra Seceleanu, University of Illinois at Urbana-Champaign, Department of Mathematics 1409 W. Green Street Urbana, IL 61801}
\email{asecele2@illinois.edu}
\address{Bart Snapp, Department of Mathematics, The Ohio State University, 100 Math Tower, 231 West 18th Avenue, Columbus, OH 43210-1174}
\email{snapp@math.ohio-state.edu}
\address{Branden Stone, Department of Mathematics, University of Kansas, 405 Snow Hall, 1460, Jayhawk Blvd, Lawrence, Kansas 66045-7594}
\email{bstone@math.ku.edu}

\begin{document}

\begin{abstract} 
We define a family of homogeneous ideals with large projective dimension and regularity relative
to the number of generators and their common degree.  This family subsumes and
improves upon constructions given in \cite{Caviglia:2004} and
\cite{McCullough}.  In particular, we describe a family of three-generated homogeneous ideals in arbitrary characteristic whose projective dimension grows asymptotically as $\sqrt{d}^{\sqrt{d} - 1}$.
\end{abstract}

\maketitle

\section{Introduction}  

Throughout this paper let $K$ be a field of any characteristic and set $R =
K[x_1,\dots,x_n]$. We consider the following question of Stillman:

\begin{qst}[Stillman, {\cite[Problem 3.14]{PS}}]\label{Q1} 
Fix a sequence of natural numbers $d_1,\ldots,d_N$.  Does there exist a number $p$ (independent of n) such that
\[\pd(R/I) \le p\]
for all graded ideals $I$ with a minimal system of homogeneous generators of degrees $d_1,\ldots,d_N$?
\end{qst}

This question is open in all but low degree cases.  In \cite{Zhang},
Zhang's work on local cohomology modules in characteristic $p$
suggested that $\sum_{i = 1}^N d_i$ was an upper bound for $\pd(R/I)$.
In \cite{McCullough}, the second author showed this was false by
producing a family of ideals whose projective dimensions far exceeded
this bound.  However, in the three-generated ideal case, these ideals
had projective dimension of only $d + 2$ where $d$ is the common
degree of the generators.  To the best of our knowledge there were no
known ideals with three degree $d$ generators with larger projective
dimension.  Clearly then $d + 2$ is a lower bound for any answer to
the three generated case of Stillman's Conjecture. We note that by the
work of \cite{Burch} and later \cite{Bruns}, it is natural to ask
whether any three-generated ideals in degree $d$ have larger
projective dimension than this.

In this paper we generalize the family of ideals given in
\cite{McCullough} to a larger family with much larger projective
dimension.  In the three-generated case, we produce a family of ideals
with generators of degree $d$ and projective dimension larger than
$\sqrt{d}^{\sqrt{d} - 1}$.  We therefore give a new lower bound for
any answer to Stillman's question.

The paper is organized as follows.  In Section~\ref{S2} we recall some
previous results and definitions.  In Section~\ref{S3} we define our family
of ideals and compute its projective dimension.  In Section~\ref{S4} we
compute some specific examples and show that this family subsumes two
interesting families of ideals previously defined.  We conclude with
some computations and questions regarding the Castelnuovo-Mumford
regularity of these ideals.

\section{Preliminaries}\label{S2}

Let $R = K[x_1,\ldots,x_n]$ and let $I = (f_1,\ldots,f_N)$ be a
homogeneous ideal and set $d_i = \deg(f_i)$.  Let $F_\*$ be the
minimal graded free resolution of $R/I$.  Then we may write
\[
F_i = \bigoplus_{j\in\Z} R(-j)^{\beta_{i,j}},
\]
where $R(-j)$ denotes a rank one free module with generator in degree
$j$.  In this case $F_0 = R$ and $F_1 = \bigoplus_{j = 1}^N R(-d_j)$.
The exponents $\beta_{i,j}$ are called the Betti numbers of $R/I$.
We can define the projective dimension of $R/I$ as
\[\pd(R/I) = \max\{i\,|\,\beta_{i,j} \neq 0 \text{ for some $j$}\}.\]  
Thus, Stillman's
question can be rephrased by asking if $\pd(R/I)$ is bounded by a
formula dependent only on $F_1$.  

The Castelnuovo-Mumford regularity of $R/I$ is defined as
\[
\reg(R/I) = \max\{j-i\,|\,\beta_{i,j} \neq 0
\text{ for some $i$}\}.
\]  
The Betti numbers are often displayed in matrix form called a Betti
table.  In the $(i,j)$ entry we put $\beta_{i,j-i}$.  Thus, we can view
the projective dimension of $R/I$ as the index of the last nonzero
column in the Betti table and the regularity of $R/I$ as the index of
the last nonzero row.  

Let $\m$ be the graded maximal ideal of $R$.  We also denote the
length of the maximal regular sequence on $R/I$ contained in $\m$ by
$\depth(R/I)$.  Finally, we let $\socle(R/I)$ denote $\{x \in
R/I\,|\,x \m = 0\}$.  To compute projective dimension, we make use of
the graded version of the Auslander-Buchsbaum Theorem (see, e.g.,
\cite[Theorem 19.9]{Eisenbud}), so in order to show that $R/I$ has
maximal projective dimension, we need only show that $\socle(R/I) \neq
0$.

Further motivating Question~\ref{Q1} is Problem 3.15 of \cite{PS} is an analog of Stillman's question for
regularity: Is there a bound for $\reg(R/I)$ dependent only on
$d_1,\ldots,d_N$?  Caviglia showed that this question is equivalent to
Question~\ref{Q1}.  See \cite{Engheta05}, pages 11--14 for a nice
explanation of this argument.

It is clear that there is an affirmative answer to Stillman's question
when $N \le 2$ or when $d_1 = \cdots = d_N = 1$.  Eisenbud and Huneke
verified the case $N = 3$ and $d_1 = d_2 = d_3 = 2$ by showing that
for ideals $I$ generated by three quadrics, $\pd(R/I) \le 4$.  In \cite{Engheta10}, Engheta
verified the case $N = 3$ and $d_1 = d_2 = d_3 = 3$ giving $\pd(R/I)
\le 36$ for this case.  This bound is likely not tight as the largest
known projective dimension of $R/I$ for an ideal $I$ generated by
three cubics is just $5$.  The first such example was found by Engheta
\cite{Engheta10}.  A simpler example is given in \cite{McCullough}.

Few other special cases of Stillman's question are known.  However, in
\cite{McCullough}, the second author defined a family of homogeneous
ideals whose projective dimension grows quickly relative to the number
and degrees of the generators.  These ideals were defined as follows:
\begin{dfn} 
Fix integers $m,n,d$ such that $m \ge 1$, $n \ge 0$ and $d \ge
2$. Order the $M_{m,d-1} = \frac{(m + d - 2)!}{(m-1)!(d-1)!}$
monomials of degree $d-1$ over the variables $x_1,\ldots,x_m$ and
denote the $i^{\text{th}}$ such monomial by $Z_i$.  Let $p =
M_{m,d-1}$ and let $R = K[x_1,\ldots,x_m,y_{1,1},\ldots,y_{p,n}]$ be a
polynomial ring in $m + pn$ variables over $K$.  We define $I_{m,n,d}$
to be the ideal generated by the following $m + n$ degree $d$
homogeneous polynomials:
\[\left\{x_i^d\,|\,1 \le i \le m\right\} \cup \left\{\sum_{j = 1}^{p} Z_jy_{j,k}\,|\,1\le k \le n\right\}.\]
\end{dfn}
It was shown that the projective dimension of $R/I$ is 
\[
\pd(R/I) = m + n p = m + n \frac{(m + d - 2)!}{(m-1)!(d-1)!}.
\]
In the three-generated degree $d$ case ($m = 2$, $n = 1$), the
projective dimension of $R/I$ is $d + 2$.  In the general case with
$N$ degree $d$ generators ($m = 2$, $n = N - 2$), the projective
dimension of $R/I$ grows asymptotically as a polynomial in $d$ of
degree $N - 2$.  In the following section we generalize this example
and define a new family of ideals with projective dimension far
exceeding both of these.

\section{A New Family of Ideals}\label{S3}

Let $K$ be a field. Fix $g\geq 2$ and integers $m_1, \ldots, m_n$ such
that $m_n \ge 0$, $m_{n-1} \ge 1$ and $m_i \ge 2$ for $1 \le i \le n -
2$. Set:
\begin{itemize}
\item $M_n = m_n$,
\item $M_k = m_k-1$ for $k<n$,
\item $d_k = m_k + \cdots + m_n + 1$,
\item $d=d_1$.
\end{itemize}
Unless explicit bounds are given,
we'll use $j$ or $j'$ for an arbitrary integer in $\{1,2,\ldots,g\}$ and $k$ or
$k'$ for an arbitrary integer in $\{1,2,\ldots,n\}$.

Finally, for $0 \le k \le n$ let
\[
\mathcal{A}_k = \left\{ (a_{j, k'}) \, \left|\, \,
\begin{minipage}{40ex}
$0\leq a_{j, k'}\leq M_{k'}$ and $\sum\limits_{j=1}^g a_{j, k'} = m_{k'}$ for $1\leq k'\leq k$, and $a_{j, k'} = 0$ for $k < k' \leq n$
\end{minipage}
\right. \right\},
\]
\medskip
\[ 
R = K[\mathsf{X}, y_{\mathsf{A}}\,|\,\,\mathsf{X} = (x_{j, k}), \mathsf{A} \in \mathcal{A}_n ], 
\]
\medskip
\[ 
I_{g, (m_1, \ldots, m_n)} = (x_{1,1}^d, \ldots, x_{g,1}^d, f), 
\]
where
\[
f = \sum_{k=1}^{n-1} \sum_{\mathsf{A}\in \mathcal{A}_{k-1}}  \sum_{j=1}^g \mathsf{X}^{\mathsf{A}}  x_{j, k}^{m_k}  x_{j, k+1} ^{d_{k+1}}
  + \sum_{\mathsf{B}\in \mathcal{A}_n} \mathsf{X}^{\mathsf{B}} y_{\mathsf{B}}.
\]
By $\mathsf{X}^\A$ we mean $\prod_{j,k} x_{j,k}^{a_{j,k}}$, where $\A =
(a_{i,j})$.  The notation above was chosen so that the monomial terms
in the generator $f$ are all of the form $\mathsf{X}^\mathsf{A}$ or
$\mathsf{X}^\mathsf{A} y_\mathsf{A}$ for some $g \times n$ matrix
$\mathsf{A}$.  We note that the restrictions on $g$ and the $m_i$
guarantee that $\mathcal{A}_i \neq \emptyset$ for all $0\le i \le n - 2$ and
$\mathcal{A}_{n-1} \neq \emptyset$ and $\mathcal{A}_n \neq \emptyset$ if and only if $m_{n-1} \ge 2$.  Before
computing the projective dimension of these ideals, we give an example
in detail.

\begin{eg}\label{e1} 
Consider the ideal $I = I_{2,(2,2,2)}$.  Then $d = d_1 = 2 + 2 + 2 + 1
= 7$, $d_2 = 2 + 2 + 1 = 5$, and $d_3 = 2 + 1 = 3$.  $M_1 = M_2 = 1$
and $M_3 = 2$.  We then have that
\[
\mathcal{A}_0 = \{\left(\begin{smallmatrix}0&0&0\\0&0&0\\\end{smallmatrix}\right)\},
\]
\[
\mathcal{A}_1 = \{\left(\begin{smallmatrix}1&0&0\\1&0&0\\\end{smallmatrix}\right)\},
\]
\[
\mathcal{A}_2 = \{\left(\begin{smallmatrix}1&1&0\\1&1&0\\\end{smallmatrix}\right)\},
\]
\[
\mathcal{A}_3 = \{\left(\begin{smallmatrix}1&1&2\\1&1&0\\\end{smallmatrix}\right), \left(\begin{smallmatrix}1&1&1\\1&1&1\end{smallmatrix}\right), \left(\begin{smallmatrix}1&1&0\\1&1&2\\\end{smallmatrix}\right) \}.
\]
Finally, the ideal $I$ is 
\[ \left(x_{1,1}^7,x_{2,1}^7,f\right),\]
where 
\begin{eqnarray*}
f &=& \X^{^{\left(\begin{smallmatrix}0&0&0\\0&0&0\\\end{smallmatrix}\right)}}x_{1,1}^2 x_{1,2}^5 + \X^{^{\left(\begin{smallmatrix}0&0&0\\0&0&0\\\end{smallmatrix}\right)}} x_{2,1}^2 x_{2,2}^5 + \X^{^{\left(\begin{smallmatrix}1&0&0\\1&0&0\\\end{smallmatrix}\right)}}x_{1,2}^2 x_{1,3}^3 + \X^{^{\left(\begin{smallmatrix}1&0&0\\1&0&0\\\end{smallmatrix}\right)}} x_{2,2}^2 x_{2,3}^3\\
&&+ \,\X^{^{\left(\begin{smallmatrix}1&1&2\\1&1&0\\\end{smallmatrix}\right)}} y_{\left(\begin{smallmatrix}1&1&2\\1&1&0\\\end{smallmatrix}\right)} + \X^{^{\left(\begin{smallmatrix}1&1&1\\1&1&1\\\end{smallmatrix}\right)}} y_{\left(\begin{smallmatrix}1&1&1\\1&1&1\\\end{smallmatrix}\right)} + \X^{^{\left(\begin{smallmatrix}1&1&0\\1&1&2\\\end{smallmatrix}\right)}} y_{\left(\begin{smallmatrix}1&1&0\\1&1&2\\\end{smallmatrix}\right)}\\
& = & {x}_{1,1}^{2} {x}_{1,{2}}^{5}+{x}_{{2},1}^{2} {x}_{{2},{2}}^{5}+{x}_{1,1} {x}_{{2},1} {x}_{1,{2}}^{2} {x}_{1,{3}}^{3}+{x}_{1,1} {x}_{{2},1} {x}_{{2},{2}}^{2} {x}_{{2},{3}}^{3} \\
&&+ \, {x}_{1,1} {x}_{{2},1} {x}_{1,{2}} {x}_{{2},{2}} {x}_{1,{3}}^{2} y_{\left(\begin{smallmatrix}1&
      1&
      {2}\\
      1&
      1&
      0\\
      \end{smallmatrix}\right)}+{x}_{1,1} {x}_{{2},1} {x}_{1,{2}} {x}_{{2},{2}} {x}_{1,{3}} {x}_{{2},{3}} y_{\left(\begin{smallmatrix}1&
      1&
      1\\
      1&
      1&
      1\\
      \end{smallmatrix}\right)}\\
      &&+ \, {x}_{1,1} {x}_{{2},1} {x}_{1,{2}} {x}_{{2},{2}} {x}_{{2},{3}}^{2} y_{\left(\begin{smallmatrix}1&
      1&
      0\\
      1&
      1&
      {2}\\
      \end{smallmatrix}\right)}.
      \end{eqnarray*}
      We note that $\mathcal{A}_2$ is not used in the definition of $I$, and in general $\mathcal{A}_{n-1}$ is not used in the definition of $I_{g,(m_1,\ldots,m_n)}$.  Moreover, $I$ is an ideal with 3 degree 7 generators in a polynomial ring $R$ with $9$ variables and $R/I$ has projective dimension $9$ by the following theorem.
\end{eg}

\begin{thm}\label{Th} Using the notation above with $I = I_{g,(m_1,\dots,m_n)}$, $\depth(R/I) = 0$.
\end{thm}

In the following proofs, we say that $\A = (a_{j,i}) \in
\mathcal{A}_k$ and $\B = (b_{j,i})\in \mathcal{A}_{k'}$ \emph{start
  the same} if $a_{j,i} = b_{j,i}$ for all $i \leq \min(k , k')$ and
all $j$ with $1 \le j \le g$.  Note that if $\A \in \mathcal{A}_0$,
then $\A$ and $\B$ start the same for all $\B \in \mathcal{A}_k$,
$0\le k \le n$.

To prove the theorem, we first need the following lemma:

\begin{lemma} For each $k$, $0 \le k \le n - 1$, let $\mathsf{E}_k = (e_{j', k'})$ be a $g\times n$ matrix where $e_{j', k'} = d_{k'} - 1$ for $1 \le j' \le g$, $1 \le k' \le k$ and zero elsewhere. Then
\[
	\mathsf{X}^{\mathsf{E}_k} x_{j, k+1}^{d_{k+1}} \in I
\]
for all $j$ such that $1 \le j \le g$ (interpret $\mathsf{E}_0 = \mathsf{0}$). 
\end{lemma}

\begin{proof}
Induct on $k$. When $k=0$, this says $x_{j, 1}^d\in I$ and indeed these are the first $g$ generators of $I$. Assume $k\geq 1$, and choose any $\mathsf{A}\in \mathcal{A}_{k-1}$. Note that $\mathsf{A}\leq \mathsf{E}_k$, so
\[
\mathsf{X}^{\mathsf{A}} x_{j, k}^{m_{k}} x_{j, k+1}^{d_{k+1}} \X^\C =  \mathsf{X}^{\mathsf{E}_k} x_{j, k+1}^{d_{k+1}}
\]
for some matrix $\C$ with nonnegative integer entries.  Notice that the matrix $\C$ is of the form
\[
	\left(\begin{array}{lllllll}
		d_1 - 1 -a_{1,1}  & \cdots & d_{k-1} - 1 - a_{1,k-1} & d_k - 1       & 0      & \cdots & 0 \\
		\phantom{move}\vdots               &        & \phantom{move over}\vdots                  & \phantom{mv}\vdots        & \,\vdots &        & \,\vdots \\
		d_1 - 1 -a_{j-1,1}  & \cdots & d_{k-1} - 1 - a_{j-1,k-1}	 & d_k - 1       & 0      & \cdots & 0 \\
		d_1 - 1 -a_{j,1} &  \cdots & d_{k-1} - 1 - a_{j,k-1} & d_k - 1 - m_k & 0      & \cdots & 0 \\
		d_1 - 1 -a_{j+1,1} &  \cdots & d_{k-1} - 1 - a_{j+1,k-1} & d_k - 1       & 0      & \cdots & 0 \\
		\phantom{move}\vdots                   &        & \phantom{move over}\vdots                  & \phantom{mv}\vdots        & \,\vdots &        & \,\vdots \\		
		d_1 - 1 -a_{g,1} &  \cdots & d_{k-1} - 1 - a_{g,k-1} & d_k - 1       & 0      & \cdots & 0\\
	\end{array}\right)\begin{array}{l}\\\\\\\\\\\\\\.\\\end{array}
\]
It is enough to show $h \X^\C \in I$ for all terms $h$ of $f$ such
that $h\neq \mathsf{X}^{\mathsf{A}} x_{j, k}^{m_{k}} x_{j,
  k+1}^{d_{k+1}}$. The remaining terms of $f$ are of the form
\[
h = \mathsf{X}^{\mathsf{B}} x_{j', k'}^{m_{k'}} x_{j', k'+1}^{d_{k'+1}}
\]
for some $\B \in \mathcal{A}_{k'-1}$ with $1 \le k' \le n - 1$ and
some $j' \le g$ such that $\A \neq \B$ or $\A = \B$ and $j' \neq j$ or
of the form
\[
h = \mathsf{X}^{\mathsf{B}} y_{\mathsf{B}}
\]
for some $\B \in \mathcal{A}_{k'}$ with $k' = n$.  Assume $h$ is one
of these terms and let $M = h\X^\C$.

If $\mathsf{A}$ and $\mathsf{B}$ do not start the same, then consider the first index $t\leq \min(k - 1, k'-1)$ where they disagree. Then the exponent of $x_{s, t}$ in $M$ will be at least $d_t$ for some $s$, and the exponents of $x_{s', t'}$ will be $d_{t'} - 1$ for all $t' < t$ (since $\mathsf{A}$ and $\mathsf{B}$ agree here), so by the inductive hypothesis, $M$ is in $I$.

Now assume that $\mathsf{A}$ and $\mathsf{B}$ start the same. We'll break this up according to cases:

{\bf Case $\boldsymbol{k' < k}$}: The exponent of $x_{j', k'}$ in $M$ is at least $d_{k'}$.  This is true since we added $d_{k'} - 1 - a_{j', k'}$ (the $(j',k')$ entry of $\C$) to $m_{k'}$ and $m_{k'} \ge a_{j', k'} + 1$. Because $\A$ and $\B$ start the same, we can write 
\[
	M = \X^{\E_{k'}}x_{j', k'}^{d_{k'}} \X^\mathsf{D}
\]
where $\mathsf{D}$ is some $g \times n$ matrix with nonnegative integral entries.  The inductive hypothesis again implies this term is in $I$.

{\bf Case $\boldsymbol{k' = k}$}: Then $\A = \B$.  Recall that $k \le n-1$ and thus $m_{k'} \ge 1$.  Since the terms defined by $\A$ and $\B$ are distinct, $j\neq j'$. So the exponent of $x_{j', k'}$  in $M$ is at least $d_{k'}$, and this term is in $I$.

{\bf Case $\boldsymbol{k' > k}$}: Notice that at least two terms in each column $k$ of $\B$ are nonzero.  This is follows when $k \le n - 2$ because $m_i \ge 2$ for $1 \le i \le n-2$.  If $k = n-1$ then $k' = n$ and $\B \in \mathcal{A}_n$.  This forces $m_{n-1} > 2$ (if $m_{n-1} = 1$ then $M_{n-1} = 0$ and $\mathcal{A}_{n-1} = \mathcal{A}_n = \emptyset$) and thus at least two terms in column $n-1$ of $\B$ are nonzero.  Now there exists some $j' \neq j$ such that $b_{j', k}$ is positive, and hence the exponent of $x_{j', k}$ in $M$ is at least $d_k$, so this term is in $I$.
\end{proof}

\begin{proof}[Proof of Theorem~\ref{Th}] We'll show that $R/I$ has depth zero by showing that the element

\[ \S = \mathsf{X}^{\mathsf{T}} \in (I:\m) - I, \]

\noindent where $\mathsf{T} = (t_{j, k})$ and $t_{j, k} = d_k - 1$; that is, the image of $\mathsf{S}$ in $R/I$ is in $\socle(R/I)$. 

Since no term of any generator of $I$ divides $\mathsf{S}$, it is clear that $\mathsf{S} \notin I$.  So we must show that every variable
multiplies $\mathsf{S}$ into $I$.  The fact that $x_{j,k} \mathsf{S} \in I$ for every $j,k$ follows from the following preceding Lemma.  We now show that $y_\mathsf{A} \S \in I$ for all $\mathsf{A} \in \mathcal{A}_n$. Notice that 
\[
	y_\mathsf{A} \S = y_\mathsf{A} \mathsf X^\mathsf{A}\cdot \mathsf X^\mathsf{C}
\]
where $\C$ is again some $g \times n$ matrix with nonnegative integral entries and $y_\mathsf{A} \mathsf X^\mathsf{A}$ is the term in $f$ associated to $y_\mathsf{A}$.  As before, it is enough to show $h\mathsf X^\mathsf{C}\in I$ for all terms $h$ in $f$ such that $h\neq y_\mathsf{A} \mathsf X^\mathsf{A}$.  Each $h$ has an $\mathsf X^\mathsf{B}$ as a factor, for some $\mathsf B \in \mathcal A_k$, $k\in \{1,2,\ldots,n-2,n\}$.

If $\mathsf A$ and $\mathsf B$ do not start the same, let $t$ be the first index where they differ. Then the exponent of some $x_{s, t}$ will be at least $d_t$ for some $s$, so by the lemma, this term is in $I$.

Otherwise, $\mathsf A$ and $\mathsf B$ start the same and $k < n-1$ (if $k = n$ then they can not start the same).  In other words, 
\[
	h = \mathsf X^\mathsf{B}x_{j,k}^{m_k}x_{j,k+1}^{d_{k+1}}.
\]
Hence $h\mathsf X^\mathsf{C}$ has $x_{j,k+1}^{d_{k+1}}\mathsf X^{\mathsf{E}_k}$ as a factor and thus, by the lemma, is an element of $I$.

\end{proof}

\begin{cor} 
\[\pd(R/I) = \prod_{i = 1}^{n-1} \left(\frac{(m_i + g - 1)!}{(g - 1)!(m_i)!} - g\right)\left(\frac{(m_n + g - 1)!}{(g - 1)!(m_n)!}\right) + g n.\]
\end{cor}

\begin{proof} This follows from the graded version of the Auslander-Buchsbaum Theorem  and by counting the number of variables in the $R$.  We get $gn$ variables $x_{j,k}$ with $1 \le j \le g$ and $1 \le k \le n$.  For each $\A \in \mathcal{A}_n$, we get a variable $y_\A$.  Note that $\mathcal{A}_n$ consists of exactly those matrices $\A$ with nonnegative integer entries such that
\begin{enumerate}
\item All the entries in column $k$ sum to $m_k$.
\item For all $k < n$, there are at least two nonzero entries in column $k$.
\end{enumerate}
In other words, the term $\prod_{j = 1}^g x_{j,k}^{a_{j,k}}$ is a monomial of degree $m_k$ in $g$ variables and when $k < n$, this monomial is not a pure power.  The formula for the projective dimension follows by counting all such monomials.
\end{proof}

\begin{eg} Continuing the notation from Example~\ref{e1}, the previous theorem shows us that 
\[\S = \X^{^{\left(\begin{smallmatrix}6&
      4&
      2\\
      6&
      4&
      {2}\\
      \end{smallmatrix}\right)}} = x_{1,1}^6 x_{2,1}^6 x_{1,2}^4 x_{2,2}^4 x_{1,3}^2 x_{2,3}^2 \in (I:\m) - I.
      \]
So the image of $\S$ in $R/I$ is in the socle of $R/I$.  It follows that $\depth(R/I) = 0$ and hence $\pd(R/I) = 9$.
\end{eg}

\begin{cor} Over any field $K$ and for any positive integer $p$, there exists an ideal $I$ in a polynomial ring $R$ over $K$ with three homogeneous generators in degree $p^2$ such that $\pd(R/I) \ge p^{p-1}$.
\end{cor}

\begin{proof}
This follows from the previous Corollary by taking the ideal 
\[
I = I_{2,({\underbrace{\scriptstyle p+1,\dots,p+1}_{p-1 \,\,\text{times}}},0)}.
\]
\end{proof}

We note that this answers two questions posed by the second author in the negative.  The following result can be viewed as an improvement to Corollary~4.7 in \cite{McCullough}.

\begin{cor} Over any field $K$ and for any positive integer $p$, there exists an ideal $I$ in a polynomial ring $R$ over $K$ with $2p+1$ homogeneous generators in degree $2p+1$ such that $\pd(R/I) \ge p^{2p}$. 
\end{cor}

\begin{proof} Take $I$ to be the ideal
\[
I_{2p,(\underbrace{\scriptstyle 2,2,2,\dots,2}_{p \,\,\text{times}})}.
\]
\end{proof}

Neither of these results gives an answer to Stillman's Question, but they impose large lower bounds on any possible answer.

\section{Examples, Special Subfamilies and Regularity Questions}\label{S4}

First we note that the family of ideals defined by the second author are a subfamily of the ideals defined above.  Using the notation in Section 2, we recall the definition for positive integers $n,d$ we define the ideal
\[ I_{m,1,d} = (x_1^d,\ldots,x_2^d,f),\]
with 
\[ f = \sum_i Z_i y_i,\]
where $Z_i$ runs through the degree $d-1$ monomials in the variables $x_1,\ldots,x_d$.  Up to relabeling of the variables, we note that
\[I_{m,1,d} = I_{m,(d)}\]
as defined in the previous section.  (We may replicate the last generator using new variables to get the full ideal $I_{m,n,d}$.)  As stated earlier, the three-generated version of these ideals satisfies $\pd(R/I) = d+2$  when the generators were taken in degree $d$.  Here we give a specific example of our new construction that improves upon this example.

\begin{eg} $I = I_{2,(3,1)}$\\

\noindent This is an ideal with 3 quintic generators such that $\pd(R/I) = 8$. \\

Let $R$ be the following polynomial ring
$$R = {K\left[{x}_{1,1},{x}_{1,{2}},{x}_{{2},1},{x}_{{2},{2}},y_{\bgroup\left( \begin{smallmatrix}{2}&
              1\\
                    1&
                          0\\
                                \end{smallmatrix}\right) \egroup},y_{\bgroup\left( \begin{smallmatrix}1&
              0\\
                    {2}&
                          1\\
                                \end{smallmatrix}\right) \egroup},y_{\bgroup\left( \begin{smallmatrix}1&
              1\\
                    {2}&
                          0\\
                                \end{smallmatrix}\right) \egroup},y_{\bgroup\left( \begin{smallmatrix}{2}&
              0\\
                    1&
                          1\\
                                \end{smallmatrix}\right) \egroup}\right]}.$$

\noindent Then the ideal $I$ is given by
\begin{align*}
I = ({x}_{1,1}^{5},{x}_{{2},1}^{5}&,{x}_{1,1}^{3} {x}_{1,{2}}^{2}+{x}_{{2},1}^{3} {x}_{{2},{2}}^{2}+{x}_{1,1}^{2} {x}_{1,{2}} {x}_{{2},1} y_{\bgroup\left( \begin{smallmatrix}{2}&
         1\\
              1&
                   0\\
                        \end{smallmatrix}\right) \egroup}+{x}_{1,1} {x}_{{2},1}^{2} {x}_{{2},{2}} y_{\bgroup\left( \begin{smallmatrix}1&
         0\\
              {2}&
                   1\\
                        \end{smallmatrix}\right) \egroup}\\
&+{x}_{1,1} {x}_{1,{2}} {x}_{{2},1}^{2} y_{\bgroup\left( \begin{smallmatrix}1&
         1\\
              {2}&
                   0\\
                        \end{smallmatrix}\right) \egroup}+{x}_{1,1}^{2} {x}_{{2},1} {x}_{{2},{2}} y_{\bgroup\left( \begin{smallmatrix}{2}&
         0\\
              1&
                   1\\
                        \end{smallmatrix}\right) \egroup})
\end{align*}
\noindent and has Betti table

 $$\begin{tabular}{c|ccccccccc}
        &0&1&2&3&4&5&6&7&8\\ \hline \hline
  \text{total:}&1&3&53&184&287&248&124&34&4 \\ \hline
  \text{0:}&1&\text{-}&\text{-}&\text{-}&\text{-}&\text{-}&\text{-}&\text{-}&\text{-}\\
  \text{1:}&\text{-}&\text{-}&\text{-}&\text{-}&\text{-}&\text{-}&\text{-}&\text{-}&\text{-}\\
  \text{2:}&\text{-}&\text{-}&\text{-}&\text{-}&\text{-}&\text{-}&\text{-}&\text{-}&\text{-}\\
  \text{3:}&\text{-}&\text{-}&\text{-}&\text{-}&\text{-}&\text{-}&\text{-}&\text{-}&\text{-}\\
  \text{4:}&\text{-}&3&\text{-}&\text{-}&\text{-}&\text{-}&\text{-}&\text{-}&\text{-}\\
  \text{5:}&\text{-}&\text{-}&\text{-}&\text{-}&\text{-}&\text{-}&\text{-}&\text{-}&\text{-}\\
  \text{6:}&\text{-}&\text{-}&\text{-}&\text{-}&\text{-}&\text{-}&\text{-}&\text{-}&\text{-}\\
  \text{7:}&\text{-}&\text{-}&\text{-}&\text{-}&\text{-}&\text{-}&\text{-}&\text{-}&\text{-}\\
  \text{8:}&\text{-}&\text{-}&3&\text{-}&\text{-}&\text{-}&\text{-}&\text{-}&\text{-}\\
  \text{9:}&\text{-}&\text{-}&3&4&\text{-}&\text{-}&\text{-}&\text{-}&\text{-}\\
  \text{10:}&\text{-}&\text{-}&13&46&68&56&28&8&1\\
  \text{11:}&\text{-}&\text{-}&33&132&218&192&96&26&3\\
  \text{12:}&\text{-}&\text{-}&1&2&1&\text{-}&\text{-}&\text{-}&\text{-}\\
  \end{tabular}$$
\end{eg}

We also note that our family of ideals subsumes another family of ideals studied by Caviglia in \cite{Caviglia:2004}.  Let $R = K[w,x,y,z]$ and let $d \ge 2$.  Then set
\[C_d = (x^d,y^d,x w^{d-1} - y z^{d-1}).\]
Caviglia showed that $\reg(R/C_d) = d^2 - 2$.  To our knowledge, this family has the fastest growing regularity relative to the degree of the generators in the three-generated case.  We note that these ideals are also a subfamily of the ideals defined in the previous section.  In fact, up to a relabeling of the variables, 
\[
C_d = I_{2,(1,d-2)}.
\]
In the following example, we show that some of our ideals have larger
regularity than Caviglia's examples.

\begin{eg} $I = I_{2,(2,1,2)}$\\

\noindent This is an ideal with 3 degree 6 generators such that $\pd(R/I) = 6$ and $\reg(R/I) = 41$. Its Betti table is displayed at the end of this section.

$$R = {K\left[{x}_{1,1},{x}_{1,{2}},{x}_{1,{3}},{x}_{{2},1},{x}_{{2},{2}},{x
      }_{{2},{3}}\right]}$$

$$I = ( x_{1,1}^6 , x_{2 ,1}^6 , x_{1,1}^2 x_{1, 2}^4 + x_{2 ,1}^2  x_{2,2}^4 + x_{1,1} x_{2 ,1}   x_{1, 2}    x_{1, 3}^3 + x_{1,1}   x_{2 ,1} x_{2 , 2}    x_{2 , 3}^3 )$$

Calculations with Macaulay2 \cite{M2} indicate that many of the ideals defined in the previous section have much larger regularity than even this example.  Specifically, we believe that the regularity of 
\[I =
I_{2,(\underbrace{\scriptstyle 2,2,2,\dots,2}_{p \,\,\text{times}},1,i)}.
\]
has regularity that grows asymptotically as $(p+2)^i$.  When $p = 0$,
this agrees with Caviglia's result.  However his methods do not extend
to the ideals above.  We note that the regularity of $R/I$ is bounded
below by the degrees of the socle elements.  However, the socle
elements we computed above only grow linearly with the degrees of the
generators.  Computing the regularity of the ideals above would
provide interesting computational examples and also give some insight
into the regularity version of Stillman's question.

Betti Table of $R/I_{2,(2,1,2)}$:\\

\pagestyle{empty}
$\begin{tabular}{c|ccccccc}
     &0&1&2&3&4&5&6\\
     \hline\hline
     \text{total:}&1&3&75&247&320&188&42\\
     \hline
     \text{0:}&1&\text{-}&\text{-}&\text{-}&\text{-}&\text{-}&\text{-}\\
     \text{1:}&\text{-}&\text{-}&\text{-}&\text{-}&\text{-}&\text{-}&\text{-}\\
     \text{2:}&\text{-}&\text{-}&\text{-}&\text{-}&\text{-}&\text{-}&\text{-}\\
     \text{3:}&\text{-}&\text{-}&\text{-}&\text{-}&\text{-}&\text{-}&\text{-}\\
     \text{4:}&\text{-}&\text{-}&\text{-}&\text{-}&\text{-}&\text{-}&\text{-}\\
     \text{5:}&\text{-}&3&\text{-}&\text{-}&\text{-}&\text{-}&\text{-}\\
     \text{6:}&\text{-}&\text{-}&\text{-}&\text{-}&\text{-}&\text{-}&\text{-}\\
     \text{7:}&\text{-}&\text{-}&\text{-}&\text{-}&\text{-}&\text{-}&\text{-}\\
     \text{8:}&\text{-}&\text{-}&\text{-}&\text{-}&\text{-}&\text{-}&\text{-}\\
     \text{9:}&\text{-}&\text{-}&\text{-}&\text{-}&\text{-}&\text{-}&\text{-}\\
     \text{10:}&\text{-}&\text{-}&3&\text{-}&\text{-}&\text{-}&\text{-}\\
     \text{11:}&\text{-}&\text{-}&\text{-}&\text{-}&\text{-}&\text{-}&\text{-}\\
     \text{12:}&\text{-}&\text{-}&\text{-}&\text{-}&\text{-}&\text{-}&\text{-}\\
     \text{13:}&\text{-}&\text{-}&2&3&\text{-}&\text{-}&\text{-}\\
     \text{14:}&\text{-}&\text{-}&\text{-}&\text{-}&\text{-}&\text{-}&\text{-}\\
     \text{15:}&\text{-}&\text{-}&\text{-}&\text{-}&\text{-}&\text{-}&\text{-}\\
     \text{16:}&\text{-}&\text{-}&3&6&3&\text{-}&\text{-}\\
     \text{17:}&\text{-}&\text{-}&\text{-}&\text{-}&\text{-}&\text{-}&\text{-}\\
     \text{18:}&\text{-}&\text{-}&1&4&5&2&\text{-}\\
     \text{19:}&\text{-}&\text{-}&4&8&4&\text{-}&\text{-}\\
     \text{20:}&\text{-}&\text{-}&1&4&6&4&1\\
     \text{21:}&\text{-}&\text{-}&2&8&10&4&\text{-}\\
     \text{22:}&\text{-}&\text{-}&6&14&11&4&1\\
     \text{23:}&\text{-}&\text{-}&2&8&12&8&2\\
     \text{24:}&\text{-}&\text{-}&4&16&21&10&1\\
     \text{25:}&\text{-}&\text{-}&8&20&18&8&2\\
     \text{26:}&\text{-}&\text{-}&3&12&18&12&3\\
     \text{27:}&\text{-}&\text{-}&6&24&32&16&2\\
     \text{28:}&\text{-}&\text{-}&3&12&18&12&3\\
     \text{29:}&\text{-}&\text{-}&4&16&24&16&4\\
     \text{30:}&\text{-}&\text{-}&3&12&18&12&3\\
     \text{31:}&\text{-}&\text{-}&4&16&24&16&4\\
     \text{32:}&\text{-}&\text{-}&1&4&6&4&1\\
     \text{33:}&\text{-}&\text{-}&4&16&24&16&4\\
     \text{34:}&\text{-}&\text{-}&1&4&6&4&1\\
     \text{35:}&\text{-}&\text{-}&2&8&12&8&2\\
     \text{36:}&\text{-}&\text{-}&1&4&6&4&1\\
     \text{37:}&\text{-}&\text{-}&2&8&12&8&2\\
     \text{38:}&\text{-}&\text{-}&1&4&6&4&1\\
     \text{39:}&\text{-}&\text{-}&2&8&12&8&2\\
     \text{40:}&\text{-}&\text{-}&\text{-}&\text{-}&\text{-}&\text{-}&\text{-}\\
     \text{41:}&\text{-}&\text{-}&2&8&12&8&2\\
     \end{tabular}$
\end{eg}

\section*{Acknowledgments}

We thank the AMS, NSF, and the organizers of this Mathematical Research Communities for making this
work possible.  We would especially like to thank Craig Huneke for several useful conversations.  Additionally, the ideals found in this note were inspired by many
Macaulay2 computations.  The interested reader should contact the authors if they would like
Macaulay2 code for investigating these ideals further.

\nocite{*}
\bibliographystyle{amsalpha} 
\bibliography{sources}

\end{document}